\documentclass[12pt]{amsart}
\usepackage{amssymb,amsmath,enumerate,verbatim}
\newtheorem{theorem}{Theorem}
\newtheorem{proposition}{Proposition}[section]
\newtheorem{lemma}[proposition]{Lemma}
\newtheorem{corollary}[proposition]{Corollary}
\newtheorem{remark}[proposition]{Remark}

\newtheorem{definition}[proposition]{Definition}

\newcommand{\R}{\mathbb{R}}
\newcommand{\C}{\mathbb{C}}
\newcommand{\Q}{\mathbb{H}}
\newcommand{\PQ}{\mathbf{P}_{\mathbb{H}}}
\newcommand{\HC}{\mathbf{H}_{\mathbb{C}}}
\newcommand{\HQ}{\mathbf{H}_{\mathbb{H}}}
\newcommand{\re}{\mathrm{Re}}
\newcommand{\im}{\mathrm{Im}}
\newcommand{\Shp}{\mathbb{S}}
\newcommand{\lp}{\langle}
\newcommand{\rp}{\rangle}
\newcommand{\tr}{\mathrm{tr}}
\newcommand{\qui}[1]{\phi({#1})}
\newcommand{\e}{\mathrm{e}}
\begin{document}
\title{Isometries of the quaternionic hyperbolic line}
\author[J. L. O. Chamorro]{Jaime L. O. Chamorro}
\address{Department of Matemathics, Federal University of Bahia, Brazil}
\email{jaime.chamorro@ufba.br}
\date{August 26, 2021}
\begin{abstract}
We give a classification of the matrices in the unitary group $U(1,1;\Q)$, where $\Q$ is the division ring of the real quaternions. To this end, we consider the complex representation $\qui{P}$ for $P\in U(1,1;\Q)$. Next, we compute the characteristic polynomial $f$ of the $4\times 4$ complex matrix $\qui{P}$ and then study the sign of the resultant of $f$ and its derivative $f'$. 
\end{abstract}
\maketitle
\section{Introduction}\label{sec1}  
The general linear group $GL(2,\Q)$ of all $2\times 2$ non-singular quaternionic matrices acts on the quaternionic projective line $\PQ^1$ by right colineations, inducing an isometric action of the unitary group $U(1,1;\Q)\subset GL(2,\Q)$ on the projective model of the quaternionic hyperbolic line $\HQ^1\subset \PQ^1$. On  \emph{ball model} or \emph{Siegel domain} of $\HQ^1$, such action is given by \emph{M\"obius trasformations} 
\begin{equation}\label{eq0}
g(x)=(ax+b)(cx+d)^{-1},\text{ for all }x\in\HQ^1,
\end{equation}
with $\begin{bmatrix}a&b\\c&d\end{bmatrix}\in U(1,1;\Q)$.

Matrices in $U(1,1;\Q)$ were classified in \cite{cao-parker-wang} by studying the fixed points of (\ref{eq0}). For that purpose, in \cite{cao-parker-wang} all solutions of the quaternionic polynomial equation $xcx+xd-ax-b=0$ are determined by solving  $t^2-(d+\bar{c}\,^{-1}b\bar{d})t+\bar{c}\,^{-1}b=0$, with $t=cx+d$.

In this work, we present an alternative approach to the classification of matrices $P\in U(1,1;\Q)$ by studing the eigenvalues of its complex representation $\qui{P}$ (see \cite {zhang, lee}), which is a $4\times 4$ complex matrix whose characteristic polynomial has the form $t^4-2\tau t^3+\rho t^2-2\tau t+1$, where $\tau$ and $\rho$ are real numbers. This approach has been inspired by \cite[section 3.5]{parker}. The main goal of the paper is the following
\begin{theorem}\label{teoA}
For any $P=\begin{bmatrix}a&b\\c&d\end{bmatrix}\in U(1,1;\Q)$, 
let  $\tau=\re(a+d)$ and $\rho=2+|c-\overline{b}|^2+4\re(a)\re(d)$. If $R_1$ and $R_2$ are the regions in the $(\tau,\rho)$-plane, defined by $4|\tau|-2\leqq \rho\leqq\tau^2+2$ and $\rho\geqq\tau^2+2$,  respectively. Then 
for $P\neq\pm I$ we have that
\begin{enumerate}
\renewcommand{\labelenumi}{(\alph{enumi})}
\item $P$ is elliptic if and only if either 
$$(\tau,\rho)\in R_1\setminus R_1\cap R_2$$ 
or 
$$(\tau,\rho)\in R_1\cap R_2\setminus \{(\pm2,6)\}\text{, }\re(a)=\re(d)\text{ and }c=\bar{b}.$$
\item $P$ is parabolic if and only if either 
$$(\tau,\rho)=(\pm2,6)$$
or 
$$(\tau,\rho)\in R_1\cap R_2\setminus \{(\pm2,6)\}\text{ and either }\re(a)\neq\re(d)\text{ or }c\neq\bar{b}.$$
\item $P$ is loxodromic if and only if $(\tau,\rho)\in R_2\setminus R_1\cap R_2$.
\end{enumerate}
\end{theorem}
This paper is organized as follows. In Section \ref{sec2}, we briefly review some basic facts on quaternions, hyperbolic quaternionic spaces and right eigenvalues of square quaternionic matrices. In Section \ref{sec3}, we prove Corollary \ref{cor1} which provides the classification of substantial right eigenvalues of matrices in $U(1,1;\Q)$. In particular, we give a geometric description of elliptic, loxodromic and parabolic isometries of the quaternionic hyperbolic line. Then, in Section \ref{sec4}, we give a proof of Theorem \ref{teoA}, by using Corollary \ref{cor1} to study parameters $\tau$ and $\rho$ in terms of substantial right eigenvalues of matrices in $U(1,1;\Q)$. Moreover, in Theorem \ref{teoB}, we give explicit expressions for the substancial right eigenvalues in terms of $\tau$ and $\rho$.
\section{Preliminares}\label{sec2}
\subsection{Quaternions}\label{sec21}
Recall that, the the division ring of the real quaternions $\Q$ is generated by symbols $i$, $j$, $k$ satisfying 
\[i^2=j^2=k^2=ijk=-1.\]
So, given $q\in\Q$, we have $q=q_0+q_1i+q_2j+q_3k$ for some $q_0$, $q_1$, $q_2$, $q_3\in\R$. Analogously to complex numbers, \emph{conjugate}, \emph{norm}, \emph{real} and \emph{imaginary parts}  of a quatenion $q$ are defined respectively by
\[\bar{q}=q_0-q_1i-q_2j-q_3k\text{, ~}|q|=\sqrt{q\bar{q}}\text{, ~}\re(q)=\dfrac{q+\bar{q}}{2}\text{~ and ~} \im(q)=\dfrac{q-\bar{q}}{2}.\] 
Also, recall that the norm is multiplicative, $q^{-1}=\frac{\bar{q}}{|q|^2}$ if $q\neq0$, and the center of $\Q$ is the field of real numbers $\R$. Moreover, a quaternion is called \emph{imaginary} if its real part is zero, and any real subalgebra of $\Q$ which is generated by a non zero imaginary quaternion is isomorphic to the field of complex numbers $\C$. Hence, we identify, as usual, $\C$ with the real subalgebra generate by $i$.

We say that $p$, $q\in\Q$ are \emph{similar} if there exists non zero $\lambda\in\Q$ such that $p=\lambda^{-1} q\lambda$. Clearly, similarity is an equivalence relation. In fact, two quaternions are similar if and only if its real parts and norms are equal \cite[lemma 1.2.2]{chen-greenberg}. Hence, any $q\in\Q$ is similar to  $\re(q)\pm|\im(q)|i\in\C$. Then for $q\in\mathbb{R}$, $q$ is similar to $p$ if and only if $q=p$ and the similarity class of $q\not\in\mathbb{R}$ contains only one complex number with positive imaginary part, since $\bar{z}=-jzj$, for all $z\in\C$.
\subsection{Quaternionic Hyperbolic Spaces}\label{sec22}
Let $\Q^{n+1}$ be the right vector space on $\Q$ of all column matrices with $n+1$ elements in $\Q$. We consider here a non-degenerate Hermitian bilinear form $\lp\cdot,\cdot\rp$ of signature $(1,\dots,1,-1)$ on $\Q^{n+1}$, and denoted by $\Q^{n,1}$ the Hermitian space $(\Q^{n+1},\lp\cdot,\cdot\rp)$. 
\begin{definition}\label{def1}
A non-zero vector $X\in\Q^{n,1}$ is called \emph{negative}, \emph{isotropic} or \emph{positive} if $\lp X,X\rp$ is negative, zero or positive, respectively. The sets of all negative, isotropic and positive vectors are denoted respectively by $V_-$, $V_0$ and $V_+$.
\end{definition}
For all non-zero $X\in\Q^{n+1}$, let $X\Q$ be the one-dimensional right vector subspace on $\Q$ generated by $X$. The set $\PQ^n$ of all $X\Q$ is the $\Q$-projective space. Negative, isotropic and positive points $X\Q\in\PQ^n$ are given by negative, isotropic and positive vectors $X$, respectively. In this way,  the \emph{projective model} of the \emph{quaternionic hyperbolic space} $\HQ^n$ is the set of all negative points in $\PQ^n$. Also, its \emph{ideal boundary} which is denoted by $\partial\HQ^n$, is given by the set of all isotropic points in $\PQ^n$. 

The unitary group $U(n,1;\Q)$ is the subgroup of all matrices $P\in M_{n+1}(\Q)$ such that 
\begin{equation}\label{eq1}
\lp PX,PY\rp=\lp X,Y\rp,~\text{ for all }~X,Y\in\Q\,^{n,1}
\end{equation}

Given a non-singular matrix $P\in M_{n+1}(\Q)$, the \emph{collineation} induced by $P$ is the (real analytic) diffeomorphism $\tilde{P}$ of $\PQ^n$ given by $X\Q\mapsto (PX)\Q$. Clearly, collineations induced by elements of $U(n,1;\Q)$ leave $\HQ^n$, $\partial\HQ^n$ and $\PQ^n-(\HQ^n\cup\partial\HQ^n)$ invariant. In fact, $U(n,1;\Q)$ acts transitively on $\HQ^n$ by isometries with kernel $\{\pm I\}$. So, $\mathbf{P}U(n,1;\Q):=U(n,1;\Q)/\{\pm I\}$ acts isometricaly and effectively on $\HQ^n$ \cite[section 2.2]{chen-greenberg}. 

Let $\Pi$ be an affine right hyperplane of $\Q^{n,1}$ which is orthogonal to some negative vector. If $\Pi\cap V_-\neq\emptyset$, we have the (real analytic) diffeomorphism $\HQ^n\to\Pi\cap V_-$, $X\Q\mapsto X\Q\cap\Pi$ which extends naturally to a homeomorphism $\HQ^n\cup\partial\HQ^n\to\Pi\cap(V_-\cup V_0)$. In fact, $\Pi\cap V_-$ is homeomorphic to an open ball in $\Pi\simeq\R^{4n}$, with boundary $\Pi\cap V_0\simeq\Shp^{4n-1}$, which is called the \emph{ball model} for $\HQ^n$. It follows from Brouwer's fixed-point theorem that any collineation has a fixed point in $\HQ^n\cup\partial\HQ^n$. In this sence we have the following
\begin{definition}\label{def2}
We say that $P\in U(n,1;\Q)$ (resp. $\tilde{P}$) is 
\begin{enumerate}
\item \emph{elliptic} if $\tilde{P}$ has a fixed point in $\HQ^n$,
\item \emph{parabolic} if $\tilde{P}$ has exactly one fixed point in $\HQ^n\cup\partial\HQ^n$ which lies on $\partial\HQ^n$,
\item \emph{loxodromic} if $\tilde{P}$ has exactly two fixed points $\HQ^n\cup\partial\HQ^n$ which belong to $\partial\HQ^n$.
\end{enumerate}
\end{definition}
In fact, Definition \ref{def2} covers all possibilities \cite[section 3.1]{chen-greenberg}.

An unbounded model for $\HQ^n$ can be constructed in a similar way to the ball model. Namely, we choose $\Pi$ to be ortogonal to some isotropic vector $X_0$. So $\Pi\cap V_-$ is a connected open unbounded set of $\Pi$ called \emph{Siegel domain}. Here, $\partial\HQ^n$ is identified with $(\Pi\cap V_0)\cup\{\infty\}$, where $\infty$ is an ideal point corresponding to $X_0\Q$. 

Finally, recall that $\HQ^1$ is called \emph{quaternionic hyperbolic line}.
\subsection{Right Eigenvalues of Quaternionic Matrices}\label{sec23}
For quaternionic matrices we have two essentially different notions of eigenvalues, namely on the left and on the right. For our purposes, we need only the second one.
\begin{definition}\label{def3} 
Given a square quaternionic matrix $P$, we say that $\lambda\in \Q$ is a \emph{right eigenvalue} of $P$ if there exists a non zero column vector $X$, such that
\begin{equation}\label{eq2}
PX=X\lambda.
\end{equation}
In this case, we will refer to $X$ as a \emph{right eigenvector} of $P$ associated to $\lambda$  
\end{definition}
On the one hand, we have that the sum of right eigenvectors is again a right eigenvector, all of them associated to the same eigenvalue. On the other hand,
a right scalar multiple of a right eigenvector is again a right eigenvector, but these are associated to similar, not necessarilly equal, eigenvalues.
In fact, it follows from Equation (\ref{eq2}) that, if $\mu\in\Q$ with $\mu\neq0$, then,
\[P(X\mu)=(X\mu)\mu^{-1}\lambda\mu.\] 
In particular, any quaternion similar to a right eigenvalue is also a right eigenvalue. So, if there exists a right eigenvalue, then necessarilly there exists at least one complex right eigenvalue. 

To find complex right eigenvalues of a square quaternionic matrix we need to asssociate	it with a certain complex matrix. In fact, note that for any $p\in\Q$ there exist unique $a$, $b\in\C$ such that $p=a+bj$. In particular, given a quaternionic matrix $P$, by applying above decomposition to each of its elements, we get unique complex matrices $A$ and $B$, such that  $P=A+Bj$. We say that 
\[\qui{P}:=\left[\begin{array}{rc}A&B\\-\bar{B}&\bar{A}\end{array}\right]\]
is the \emph{complex representation} of $P$. We will see below, in addition to other properties that, any eigenvalue of $\qui{P}$ is also a right eigenvalue of $P$. 

The next four propositions, have been proved in \cite[Sec. 4, 5, 6]{zhang} and \cite{lee}.
\begin{proposition}\label{prop1}
We have that $\chi:M_{m}(\Q)\to M_{2m}(\C)$, $P\mapsto\qui{P}$, is an injective homomorphism of real algebras such that
\begin{enumerate}
\item $\qui{P^{\ast}}=\qui{P}^{\ast}$;
\item $\det(\qui{P})\geqq0$.
\end{enumerate}
\end{proposition}
\begin{proposition}\label{prop2}
Given $P\in M_m(\Q)$, we have that, $\lambda\in\C$ is a right eigenvalue of $P$ if and only if $\lambda$ is an eigenvalue of $\qui{P}$. In particular, there exists at least one and at most $m$ distinct similarity classes of right eigenvalues for $P$. 
\end{proposition}
\begin{proposition}\label{prop3}
If $P\in M_m(\Q)$, then characteristic polynomial of its complex representation $\qui{P}$ has the form
\[(t-\lambda_1)(t-\overline{\lambda_1})(t-\lambda_2)(t-\overline{\lambda_2})\cdots(t-\lambda_{m})(t-\overline{\lambda_{m}}),\]
where either $\lambda_{\alpha}\in\R$ or $\lambda_{\alpha}$ has positive imaginary part, for all $\alpha=1,2,\dots,m$. 
\end{proposition}
\begin{definition}\label{def5}
Complex numbers $\lambda_1,\dots,\lambda_m$ as in Propositon \ref{prop3} are called \emph{substantial right eigenvalues} of $P$.
\end{definition}
\begin{proposition}\label{prop4}
Given $P\in M_2(\Q)$, let $\lambda_1$ and $\lambda_2$ be its substantial right eigenvalues. We have that $\qui{P}$ is similar to $\begin{bmatrix}J&0\\0&\bar{J}\end{bmatrix}$, where either $J=\begin{bmatrix}\lambda_1&0\\0&\lambda_2\end{bmatrix}$ or $J=\begin{bmatrix}\lambda&0\\1&\lambda\end{bmatrix}$, with $\lambda=\lambda_1=\lambda_2$. Moreover, $P$ is similar to $J$ and hence $P$ is diagonalizable if and only if $\qui{P}$ is diagonalizable.
\end{proposition}
We end this section by giving a sufficient condition for a $2\times 2$ quaternionic matrix to be diagonalizable.
\begin{corollary}\label{newcor}
Given $P\in M_2(\Q)$, let $\lambda_1$ and $\lambda_2$ be its substantial right eigenvalues. If $\lambda_1\neq\lambda_2$, then $P$ is diagonalizable.
\end{corollary}
\begin{proof}
For $\alpha=1$, $2$, let $X_{\alpha}$ be a right eigenvector of $P$ associated to $\lambda_{\alpha}$. Consider $\qui{X_{\alpha}}$ and denote by $\xi_{\alpha}$ and $-\eta_{\alpha}$ its columns. Clearly, $\xi_{\alpha}$ and $\eta_{\alpha}$ are linearly independent eigenvectors of $\qui{P}$, associated to $\lambda_{\alpha}$ and $\overline{\lambda_{\alpha}}$, respectively. If $\lambda_1\neq\lambda_2$, then $\lambda_1$ and $\lambda_2$ are not similar, and hence $X_1$ and $X_2$ are right linearly independent. In particular, $\{\xi_1,\eta_1,\xi_2,\eta_2\}$ is linearly independent, since $a_1\xi_1+b_1\eta_1+a_2\xi_2+b_2\eta_2=0$ if and only if $X_1(a_1+\overline{b_1}j)+X_2(a_2+\overline{b_2}j)=0$.
\end{proof}
\section{The group $U(1,1;\Q)$}\label{sec3}
As usually, for $X$, $Y\in\Q^{1,1}$ (see Section \ref{sec22}) we put $\lp X,Y\rp=X^{\ast}\mathbb{J}Y$, where $\mathbb{J}=\left[\begin{array}{cr}1&0\\0&-1\end{array}\right]$. Thus, $P\in U(1,1;\Q)$ if and only if
\begin{equation}\label{eq3}
P^{\ast}\mathbb{J}P=\mathbb{J}.
\end{equation}
\subsection{Substantial Right Eigenvalues}
Now we will determine the substantial rigth eigenvalues of matrices in $U(1,1;\Q)$.
\begin{lemma}\label{lema1}
If $P\in U(1,1;\Q)$ then
\begin{enumerate}[(i)]
\item $\det(\qui{P})=1$,
\item $\lambda$ is an eigenvalue of $\qui{P}$ if and only if $\bar{\lambda}^{-1}$ is an eigenvalue of $\qui{P}$.
\end{enumerate}
\end{lemma}
\begin{proof}
First, we apply $\chi$ to Equation (\ref{eq3}). For (i), use Proposition \ref{prop1}. For (ii), we have that $\qui{P}$ and $\qui{P}^{\ast~-1}$ are similar matrices since $\qui{\mathbb{J}}^2=I$.
\end{proof}
As a direct consequence of Proposition \ref{prop3} and Lemma \ref{lema1} we get the following
\begin{corollary}\label{cor1}
Given $P\in U(1,1;\Q)$ let $\lambda_1$ and $\lambda_2$ be the substantial right eigenvalues of $P$. We have one and only one of the following possibilities
\begin{enumerate}[(i)]
\item $\lambda_1=\lambda_2=\pm1$;
\item $\lambda_1=1$ and $\lambda_2=-1$;
\item $\lambda_1=r$ and $\lambda_2=\frac{1}{r}$, for some $r\in\R$, $r\neq-1,0,1$;
\item $\lambda_1=\pm1$ and $\lambda_2=\e^{i\theta}$ for some $\theta\in]0,\pi[$;
\item $\lambda_1=\lambda_2=\e^{i\theta}$ for some $\theta\in]0,\pi[$;
\item $\lambda_1=\e^{i\theta_1}$ and $\lambda_2=\e^{i\theta_2}$ for some $\theta_1$, $\theta_2\in]0,\pi[$, $\theta_1\neq\theta_2$;
\item $\lambda_1=r\e^{i\theta}$ and $\lambda_2=\frac{1}{r}e^{i\theta}$ for some $r>0$, $r\neq1$ and $\theta\in]0,\pi[$.
\end{enumerate}
\end{corollary}
\subsection{The action of $U(1,1;\Q)$ on quaternionic hyperbolic line}
The following two lemmas will be usefull to describe the isometries of $\HQ^1$ induced by $U(1,1;\Q)$.
\begin{lemma}\label{lema2}
Let $\lambda$, $\mu$ be right eigenvalues of $P\in U(1,1;\Q)$ and  let $X$, $Y$ be right eigenvectors of $P$ associated to $\lambda$ and $\mu$, respectively. Then
\begin{enumerate}[(i)]
\item If $|\lambda|\neq1$, then $\lp X,X\rp=0$ .
\item For $\lambda,\mu\in\C$, if $\bar{\lambda}\mu\neq1$ and $\lambda\mu\neq1$, then $\lp X, Y\rp=0$.
\end{enumerate}
\end{lemma}
\begin{proof}
From Equation (\ref{eq1}), we have that $\lp X,Y\rp=\lp PX,PY\rp=\lp X\lambda,Y\mu\rp=\bar{\lambda}\lp X,Y\rp\mu$. Making $X=Y$ and $\lambda=\mu$ we get (i). Now, taking $a$, $b\in\C$ such that $\lp X,Y\rp=a+bj$, we get $a=\bar{\lambda}\mu a$, $b=\overline{\lambda\mu}b$ and (ii) follows.
\end{proof}
\begin{lemma}[\cite{chen-greenberg}, Prop. 2.1.4]\label{lema3}
If $X\in\Q^{1,1}$ then
\begin{enumerate}[(i)]
\item $X\in V_-$ if and only if $X^{\perp}\subset V_+$.
\item $X\in V_+$ if and only if $X^{\perp}\subset V_-$.
\item $X\in V_0$ if and only if $X^{\perp}=X\Q$.
\end{enumerate}
\end{lemma}
\subsubsection{Elliptic Case}\label{elliptic}
Let $P\in U(1,1;\Q)$ as in Corollary \ref{cor1} (ii), (iv), (v) and (vi), where for case (v) we assume that $P$ is diagonalizable (see Propositon \ref{prop4}). For cases (ii), (iv) and (vi), Lemma \ref{lema2} (ii) implies that right eigenvectors of $P$ associated to $\lambda_1$ and $\lambda_2$ are orthogonal. So, from Lemma \ref{lema3} there exists a right orthonormal basis $\{e_1,e_2\}$ of $\Q^{1,1}$ with $Pe_1=e_1\lambda_1$ and $Pe_2=e_2\lambda_2$. Note that this is also satisfied by case (v). In fact, as $\lambda_1=\lambda_2=\lambda$ then $\lp e_1,e_2\rp\in\C$, since $\lp e_1,e_2\rp=\lp Pe_1,Pe_2\rp=\bar{\lambda}\lp e_1,e_2\rp\lambda$. So, $W=e_1\C\oplus e_2\C$ is a \emph{$\C$-hyperbolic subspace of $\Q^{1,1}$} (see \cite[Sec. 2.1]{chen-greenberg}) with $Pw=w\lambda$, for any $w\in W$. In particular, we can suppose that $\{e_1,e_2\}$ is an orthonormal basis of $W$, and hence an orthonormal right basis for $\Q^{1,1}$. 

So, if $X=e_1X_1+e_2X_2$, we get $\lp X, X\rp=|X_1|^2-|X_2|^2$ and $PX=e_1\lambda X_1+e_2\lambda X_2$. Thus, $x=X_1X_2^{-1}$ gives the ball model $|x|<1$ for $\HQ^1$ and the isometry induced by $P$ is an euclidean motion (see \cite{porteous}), namely 
\[x\mapsto\lambda_1x\lambda_2^{-1}=\begin{cases}
                                     -x,&\text{ if }\lambda_1=1,~\lambda_2=-1\\
                                     \pm x\e^{-i\theta},&\text{ if }\lambda_1=\pm1,~\lambda_2=\e^{i\theta}\\
                                     \e^{i\theta}x\e^{-i\theta},&\text{ if }\lambda_1=\lambda_2=\e^{i\theta}\\
                                     \e^{i\theta_1}x\e^{-i\theta_2},&\text{ if }\lambda_1=\e^{i\theta_1},~\lambda_2=\e^{i\theta_2}\end{cases}.\]
Moreover, if $x=u+vj$, $u$, $v\in\C$, then $\lambda_1x\lambda_2^{-1}=\e^{i\alpha}u+\e^{i\beta}vj$, where 
\[(\alpha,\beta)=\begin{cases}
                                     (\pi,\pi)\\
                                     (-\theta,\theta),~(\pi-\theta,\pi+\theta)\\
                                     (0,2\theta)\\
                                     (\theta_1-\theta_2,\theta_1+\theta_2)\end{cases}                                      
,\]
respectively. Note that, $v=0$ and $u=0$ are two totally geodesic surfaces \cite[section 2.5]{chen-greenberg} (in fact for case (v), $v=0$ corresponds to the projectivization of $W$) which are both isometric to the complex hyperbolic line $\HC^1$ and intersects orthogonally at a single point, namely at $x=0$. So, in all cases except (v) $P$ acts by two simultaneous rotations on $v=0$ and $u=0$ with unique fixed point $x=0$. In case (v),  $P$ induces an isometry whose restriction to surface $v=0$ is the identity and acts by a rotation on the surface $u=0$. 
\subsubsection{Parabolic Case}\label{parabolic}
Let $P\in U(1,1;\Q)$ as in Corollary \ref{cor1} (i) and (v) assuming that $P$ is not diagonalizable. Let $\lambda$ be the unique substantial eigenvalue of $P$. From Propositon \ref{prop4} we can take a right basis $\{e_1,e_2\}$ of $\Q^{1,1}$ such that $Pe_1=e_1\lambda$ and $Pe_2=e_1+e_2\lambda$. Then
\begin{equation}\label{eq4}
\begin{aligned}
\lp e_1,e_2\rp =& \bar{\lambda}\lp e_1,e_1\rp+\bar{\lambda}\lp e_1,e_2\rp\lambda\\
\lp e_2,e_2\rp =& \lp e_1,e_1\rp+2\re(\lp e_1,e_2\rp\lambda)+\lp e_2,e_2\rp.
\end{aligned}
\end{equation}
By first and second equations in (\ref{eq4}) we get, respectively, that $e_1\in V_0$ and hence $\re(\lp e_1,e_2\rp\lambda)=0$. Now, by setting $\alpha=\lp e_1,e_2\rp$, (\ref{eq4}) entails that 
\begin{equation*}
\lambda\alpha=\alpha\lambda\text{~ and ~}\re(\alpha\lambda)=0,
\end{equation*}
where $\alpha\neq0$, by Lemma \ref{lema3} (iii). So, setting $X=e_1X_1+e_2X_2$, we get 
\[\lp X, X\rp= r|X_2|^2+2\re(\overline{X_1}\alpha X_2)\text{~ and ~}PX=e_1(\lambda X_1+X_2)+e_2\lambda X_2,\]
with $r=\lp e_2,e_2\rp$. Then, by making $\xi=X_1X_2^{-1}$, we obtain the Siegel domain $r+\re(\bar{\alpha}\xi)<0$ for $\HQ^1$, and the isometry induced by $P$ is given by 
\[\xi\mapsto\lambda\xi\lambda^{-1}+\lambda^{-1}=\begin{cases}\xi\pm 1,\text{ if }\lambda=\pm1\\\e^{i\theta}\xi\e^{-i\theta}+\e^{-i\theta},\text{ if }\lambda=\e^{i\theta}\end{cases}.\]
In the first case, we have a Heisenberg translation. In second case, we have the composition of a Heisenberg translation  with a rotation (as in elliptic case (v)). In two cases, the isometry induced by $P$ fixes only $\infty$.
\subsubsection{Loxodromic Case}\label{loxodromic}
Let $P\in U(1,1;\Q)$ as in Corollary \ref{cor1} (iii), (vii). From Corollary \ref{newcor} and Lemma \ref{lema2} (i), there exits a right basis $\{e_1,e_2\}\subset V_0$ of $\Q^{1,1}$ with $Pe_1=e_1\lambda_1$ and $Pe_2=e_2\lambda_2$. By Lemma \ref{lema3} (iii), $\lp e_1,e_2\rp\neq0$, so by substituting $e_1$ with $e_1\overline{\lp e_1,e_2\rp}\,^{-1}$, if necessary, we can suppose that $\lp e_1,e_2\rp=1$. Now, by taking $X=e_1X_1+e_2X_2$, we get $\lp X, X\rp=2\re(\overline{X_1}X_2)$ and $PX=e_1\lambda_1X_1+e_2\lambda_2X_2$. Thus, if $\xi=X_1X_2^{-1}$, Siegel domain for $\HQ^1$ is given by $\re(\xi)<0$ and the isometry induced by $P$ is
\[\xi\mapsto\lambda_1\xi\lambda_2^{-1}=\begin{cases}r^2\xi,\text{ if }\lambda_1=r,~\lambda_2=\frac{1}{r}\\r^2e^{i\theta}\xi\e^{-i\theta},\text{ if }\lambda_1=r\e^{i\theta},~\lambda_2=\frac{1}{r}\e^{i\theta}\end{cases},\]
which preserves Riemannian geodesic with end points $0$ and $\infty$ and fixes only these two points. In the first case, we have an homothety. In the second case, we have composition of an homothety with a rotation (as in elliptic case (v)). 
\section{Proof of Theorem \ref{teoA}}\label{sec4}
From Equation (\ref{eq3}) we have that $P=\begin{bmatrix}a&b\\c&d\end{bmatrix}\in U(1,1;\Q)$ if and only if 
\begin{equation*}
|a|^2-|b|^2=1,\text{ } \bar{a}c-\bar{b}d=0~\text{ and } |c|^2-|d|^2=-1.
\end{equation*}
Now, consider the complex representation $\qui{P}$ of $P$. On the one hand, by a direct computation, we have that characteristic polynomial $f$ of $\qui{P}$ is
\begin{equation}\label{eq5}
f=t^4-2\tau t^3+\rho t^2-2\tau t+1,
\end{equation}
where $2\tau=\tr(\qui{P})=																																																																																2\re(a+d)$ and $\rho=2+|c-\overline{b}|^2+4\re(a)\re(d)$. On the other hand, from Proposition \ref{prop3}, we get
\begin{equation}\label{eq6}
f=(t-\lambda_1)(t-\overline{\lambda_1})(t-\lambda_2)(t-\overline{\lambda_2}),
\end{equation}
where $\lambda_1$ and $\lambda_2$ are the substantial right eigenvalues of $P$. So from equations (\ref{eq5}), (\ref{eq6}) and Lemma \ref{lema1} we have
\begin{equation}\label{eq7}
\tau=\re(\lambda_1+\lambda_2)~\text{ and }~\rho=|\lambda_1|^2+|\lambda_2|^2+4\re(\lambda_1)\re(\lambda_2).
\end{equation}
Thus, from Corollary \ref{cor1} one readily gets the following
\begin{corollary}\label{cor2}
For any $P\in U(1,1;\Q)$ let $\tau$ and $\rho$ be as in (\ref{eq7}). We have one, and only one, of the following possibilities:
\begin{enumerate}[(i)]
\item $\tau=\pm 2$ and $\rho=6$;
\item $\tau=0$ and $\rho=-2$;
\item $\tau=r+\frac{1}{r}$ and $\rho=r^2+\frac{1}{r^2}+4$; 
\item $\tau=\pm 1+\cos\theta$ and $\rho=2\pm 4\cos\theta$; 
\item $\tau=2\cos\theta$ and $\rho=2+4\cos^2\theta$;
\item $\tau=\cos\theta_1+\cos\theta_2$ and $\rho=2+4\cos\theta_1\cos\theta_2$;
\item $\tau=\left(r+\frac{1}{r}\right)\cos\theta$ and $\rho=r^2+\frac{1}{r^2}+4\cos^2\theta$.
\end{enumerate}
Where $r\in\R$ with $r\neq-1,0,1$ and $\theta$, $\theta_1$, $\theta_2\in]0,\pi[$ wiht $\theta_1\neq\theta_2$.
\end{corollary}
Now, recall that to know if a polynomial $f$ has a repeated root it is enough to know if $f$ and its derivative $f'$ have a common root. For this, we can compute the \emph{resultant} of $f$ and $f'$ \cite[Lemma 3.3]{kirwan}. 
\begin{lemma}
For any $P\in U(1,1;\Q)$, characteristic polynomial $f$ of $\qui{P}$ (see Equation (\ref{eq5})) has a repeated root if and only if
\[\left|\begin{array}{rrrrrrr}1&-2\tau&\rho&-2\tau&1&0&0\\0&1&-2\tau&\rho&-2\tau&1&0\\0&0&1&-2\tau&\rho&-2\tau&1\\4&-6\tau&2\rho&-2\tau&0&0&0\\0&4&-6\tau&2\rho&-2\tau&0&0\\0&0&4&-6\tau&2\rho&-2\tau&0\\0&0&0&4&-6\tau&2\rho&-2\tau\end{array}\right|=0,\]
\end{lemma}
i.e. 
\[(\rho+4\tau+2)(\rho-4\tau+2)(\rho-\tau^2-2)^2=0.\]
Note that from Corollary \ref{cor1}, $(\rho+4\tau+2)(\rho-4\tau+2)(\rho-\tau^2-2)^2\neq 0$ only in cases (vi) and (vii). It follows that the factors $\rho\pm 																																																																																																																																																																																																																																																																																	4\tau+2$, and $\rho-\tau^2-2$ give important discriminants, which are our main tool to find the class of P. In fact, in the $(\tau,\rho)$-plane consider the curves $\rho=4|\tau|-2$ and $\rho=\tau^2+2$. These curves determine the regions $R_1$ and $R_2$ as in Theorem \ref{teoA}. So, from Corollary \ref{cor2} we get
\begin{corollary}\label{cor3}
Seven cases in Corollary \ref{cor1} are distributed as follows
\begin{enumerate}[(a)]
\item For cases (ii), (iv) and (vi), one has that $(\tau,\rho)\in R_1-R_1\cap R_2$.
\item For cases (i) and (v), one has that $(\tau,\rho)\in R_1\cap R_2$.
\item For cases (iii) and (vii), one has that $(\tau,\rho)\in R_2- R_1\cap R_2$.
\end{enumerate}
\end{corollary}
\begin{remark}\label{rk1}
$P\in U(1,1;\Q)$ is called \emph{simple} if $P$ is similar in $U(1,1;\Q)$ to a complex matrix. When $P$ is not simple, we say that $P$ is \emph{compound}. One has that \cite[Prop. 1.2]{cao-parker-wang} $P=\begin{bmatrix}a&b\\c&d\end{bmatrix}$ is simple if and only if
$$\re(a)=\re(d)~~~c=\bar{b}.$$ Now, note that $\rho=\tau^2+2$ (resp. $\rho>\tau^2+2$, $\rho<\tau^2+2$) if and only if $|\bar{b}-c|^2=\re(a-d)^2$ (resp. $|\bar{b}-c|^2>\re(a-d)^2$, $|\bar{b}-c|^2<\re(a-d)^2$). So, simple elements of $U(1,1;\Q)$ appear only along the parabola $\rho=\tau^2+2$. However, taking $d=-\bar{a}$ and $c=-b$ with $b=\re(a)\neq0$ and $|\im(a)|=1$, we get that $P$ is compound and parabolic, with $(\tau,\rho)=(0,2)$ being the vertex of the parabola. 
\end{remark}
\begin{proof}[Proof of Theorem \ref{teoA}]
According to the discussion in sections  \ref{elliptic}, \ref{parabolic} and \ref{loxodromic}, cases (a) and (c) in Corollary \ref{cor3} correspond, respectively, to elliptic and loxodromic isometries. While in case (b) we have both elliptic and parabolic isometries. So, to finish the proof, it is enough to analize when matrix $P$ is diagonalizable in case (b). This case correspond to $\lambda_1=\lambda_2=\lambda$ and $|\lambda|=1$.
On the one hand, if $\lambda=\pm 1$, $P$ is diagonalizable if and only if $P=\pm I$. On the other hand if $\lambda\not\in\R$ and $|\lambda|=1$, from Proposition \ref{prop4}, we get that $P$ is diagonalizable if and only if minimal polynomial of $\qui{P}$ is 
\[(t-\lambda)(t-\bar{\lambda})=t^2-2\re(\lambda)t+1.\]
In this case,
\begin{equation*}
\qui{P}^2-2\re(\lambda)\qui{P}+I=0.
\end{equation*}
From Proposition \ref{prop1}, the last equation means $P^2-2\re(\lambda)P+I=0$, or, equivalently
\[P+P^{-1}=2\re(\lambda)I.\]
But, if $P=\begin{bmatrix}a&b\\c&d\end{bmatrix}$, then $P^{-1}=\mathbb{J}P^{\ast}\mathbb{J}=\left[\begin{array}{rr}\bar{a}&-\bar{c}\\-\bar{b}&\bar{d}\end{array}\right]$ from Equation (\ref{eq3}). Therefore, $P$ is diagonalizable if and only if $\re(a)=\re(d)=\re(\lambda)$ and $c=\bar{b}$.  
\end{proof}
Finally, as a byproduct one has explicit formulas for the substantial eigenvalues in terms of $\tau$ and $\rho$.
\begin{theorem}\label{teoB}
Substancial right eigenvalues of $P\in U(1,1;\Q)$ are given by 
\begin{enumerate}
\renewcommand{\labelenumi}{(\alph{enumi})}
\item $\lambda_1=\mathrm{e}^{i\arccos\frac{\tau+\sqrt{\tau^2+2-\rho~}}{2}}$, $\lambda_2=\mathrm{e}^{i\arccos\frac{\tau-\sqrt{\tau^2+2-\rho~}}{2}}$, if $(\tau,\rho)\in R_1$;
\item $\lambda_1=r\mathrm{e}^{i\theta}$,  $\lambda_2=\frac{1}{r}\mathrm{e}^{i\theta}$, where $\theta=\arccos\sqrt{\frac{\rho+2-\sqrt{(\rho+2)^2-16\tau^2~}}{8}~}$ and $r=\frac{\tau+\sqrt{\tau^2-4\cos^2\theta~}}{2\cos\theta}$, if $(\tau,\rho)\in R_2-R_1\cap R_2$.
\end{enumerate}
\end{theorem} 
\begin{proof}
We can divide the seven cases in Corollary \ref{cor2} in two classes. Namely,
for (i), (ii), (iv), (v) and (vi) we have that $(\tau,\rho)\in R_1$ and
\begin{equation}\label{eq8}
\tau=\cos\theta_1+\cos\theta_2\text{ and }\rho=2+4\cos\theta_1\cos\theta_2.
\end{equation}
For (iii) and (vii) we get $(\tau,\rho)\in R_2-R_1\cap R_2$ and
\begin{equation}\label{eq9}
\tau=\left(r+\frac{1}{r}\right)\cos\theta\text{ and }\rho=r^2+\frac{1}{r^2}+4\cos\theta.
\end{equation} 
So, the claim readily follows from equations (\ref{eq8}) and (\ref{eq9}), by solving respectively  for $(\theta_1,\theta_2)$ and $(r,\theta)$.
\end{proof}

\end{document}